\documentclass[11pt]{amsart}
\usepackage{verbatim, latexsym, amssymb, amsmath,color}
\usepackage{epsfig}
\usepackage{hyperref}
\usepackage{float}
\usepackage{graphicx}
\usepackage{subfig}
\usepackage{tikz}
\usepackage{setspace}
\newtheorem{thm}{Theorem}[subsection]
\newtheorem{lemm}[thm]{Lemma}

\newtheorem*{thmA}{Theorem A (The Ricci condition)}
\newtheorem*{thmB}{Theorem B (Osserman-do Carmo-Dajczer condition)}
\newtheorem*{thmC}{Theorem C}
\newtheorem*{thmD}{Theorem D}
\newtheorem*{thmE}{The free boundary correspondence}

\theoremstyle{remark}

\theoremstyle{definition}
\newtheorem{exam}[thm]{Example}
\newtheorem*{defi}{Definition}
\newtheorem{rmk}[thm]{Remark}

\numberwithin{equation}{subsection}
\allowdisplaybreaks[1]

\title[Intrinsic metric of CMC surfaces with free boundary]{The intrinsic metric of constant mean curvature surfaces and minimal hypersurfaces with free boundary}
\author{Lucas Ambrozio}
\address{L. Ambrozio: IMPA \\ Rio de Janeiro RJ 22460-320
Brazil}
\email{l.ambrozio@impa.br}

\thanks{L.A. is supported by CNPq - Conselho Nacional de Desenvolvimento Cient\'ifico e Tecnol\'ogico (309908/2021-3 - Bolsa PQ) and by FAPERJ - Funda\c{c}\~ao Carlos Chagas Filho de Amparo \`a Pesquisa do Estado do Rio de Janeiro (grant SEI-260003/000534/2023 - BOLSA E-26/200.175/2023 and grant SEI-260003/001527/2023 - APQ1 E-26/210.319/2023)}

\begin{document}

\begin{abstract}
	Ricci-Curbastro established necessary and sufficient conditions for a Riemannian metric on a surface to be the first fundamental form of a minimal immersion of that surface into the Euclidean space. We revisit certain developments arising from his theorem, and propose new versions of these results in the context of the theory of constant mean curvature surfaces in three-dimensional space forms that meet umbilical surfaces orthogonally along their boundary components. Higher dimensional generalisations, inspired by a theorem of do Carmo and Dajczer, are discussed as well.
\end{abstract}

\maketitle

{\centering\footnotesize \textit{Dedicated to Professor Marcos Dajczer on the occasion of his 75th birthday}.\par}

\hfill
\vspace{2mm}

	In 1981, Marcos Dajczer and Manfredo do Carmo published a short article on the \textit{Boletim da Sociedade Brasileira de Matemática} \cite{dCDaj}. The main result of that work was a higher dimensional generalisation, for minimal hypersurfaces, of classical theorems about minimal surfaces in three-dimensional space forms. All these results trace back to a theorem of Ricci-Curbastro about the first fundamental form of minimal surfaces in $\mathbb{R}^3$ \cite{Ric}.
	
	We want to revisit these theorems with the aim of connecting them to a topic in the theory of minimal surfaces that has attracted great attention in recent years, namely, free boundary minimal surfaces in Euclidean balls \cite{Li}. On the way, we establish a correspondence between constant mean curvature surfaces in three-dimensional space forms that meet umbilical surfaces in right angles along their boundary components, and we prove the analogue of do Carmo-Dajczer main theorem in \cite{dCDaj} under this sort of boundary condition. 
	
	We have not attempted to exhaust the subject. Rather, we decided to focus on explaining a couple of key ideas and on discussing how they fit in current research. In any case, our expectation is that the systematic study of the ideas sketched here only briefly may bear fruits in the future.
	
	It is a pleasure to congratulate Professor Marcos Dajczer on this special occasion!		
	
	\section{The first fundamental form of CMC immersions}

\subsection{Motivation} Let $(\Sigma^2,g)$ be a connected, oriented Riemannian surface without boundary, not necessarily complete. The basic problem in the classical theory of isometric immersions is to identify properties of $g$ that can determine whether or not $(\Sigma^2,g)$ admits an isometric immersion into the Euclidean space $\mathbb{R}^3$, or, more generally, into one of the (simply connected) space forms $\mathbb{Q}^3(c)$ of constant sectional curvature $c\in \mathbb{R}$. 

	By the Fundamental Theorem of Isometric Immersions (see Theorem 1.11 in \cite{DajToj}), passing to the universal cover of $\Sigma^2$, the problem boils down to determine conditions which allow a symmetric bilinear form $A$ to be constructed on $\Sigma^2$, out of $g$ and invariant by deck transformations, in such way that the pair $(g,A)$ satisfies the two fundamental equations:
\begin{align*}
	\textit{(Gauss equation)}: & \quad K = c + det(A), \quad \text{and} \\
	\textit{(Coddazi equation)}: &  \quad (\nabla_X A)(Y,Z)= (\nabla_Y A)(X,Z) \text{ for all } X,Y,Z\in \mathcal{X}(\Sigma),
\end{align*}
and moreover the right period conditions are met. In the above formulas, $K$ denotes the Gaussian curvature and $\nabla$ denotes the Levi-Civita connection of $(\Sigma^2,g)$, while $det(A)$ denotes the product of the two eigenvalues of $A$. 

	If such tensor $A$ is found, then $A$ is the second fundamental form of an isometric immersion of $(\Sigma^2,g)$ into $\mathbb{Q}^{3}(c)$, which is unique up to orientation preserving isometries of $\mathbb{Q}^{3}(c)$. (Given the chosen orientations of $\Sigma^2$ and $\mathbb{Q}^3(c)$, a unique unit normal vector field $N$ is determined and, up to sign, $A=\langle D_- N,-\rangle$, where $D$ is the Levi-Civita connection of $\mathbb{Q}^3(c)$). 

	However, the two fundamental equations form a difficult and subtle geometric PDE system for $A$, which is tricky to address directly.
	
	 A more approachable problem is, perhaps, to look for conditions that are necessary and sufficient for the metric $g$ to arise from an isometric immersion of $(\Sigma^2,g)$ into  $\mathbb{Q}^3(c)$ that satisfies, moreover, further meaningful geometric restrictions. For instance, one could naively ask the following basic question: 
	 
	 \begin{center}
	 	\textit{What is special about the metric $g$ if it is the first fundamental form of a \textit{constant mean curvature} immersion of $\Sigma^2$ into $\mathbb{Q}^3(c)$?}
	 \end{center} 
	 Recall that an isometric immersion has constant mean curvature when $H$ - the \textit{average} of the trace of the second fundamental form $A$ with respect to $g$ - is constant on $\Sigma^2$.
	
	The same question can be formulated, of course, also in higher dimensions and codimensions. But in that degree of generality, interesting necessary conditions are less easy to find, and even more so sufficient conditions. See \cite{CheOss} for a comprehensive discussion on the case of minimal submanifolds of the Euclidean space $\mathbb{R}^n$. 
	
	Since in this article we want to discuss only the case of hypersurfaces, we therefore return to the formulation we started with, turning our attention back to the case of immersed surfaces in $\mathbb{Q}^3(c)$.	
	
\subsection{Dimension two} \label{sec1.2} A rather satisfactory answer to the basic question has been known in the case of \textit{minimal} ($H=0$) surfaces in $\mathbb{R}^3$ since at least the eighteen hundreds. The original theorem of Ricci-Curbastro \cite{Ric} was generalised by Lawson in the case of constant mean curvature (CMC) surfaces in the Euclidean sphere $\mathbb{S}^3=\mathbb{Q}^{3}(1)$ (see \cite{Law}, Theorem 8). The generalisation for constant mean curvature surfaces in all three-dimensional space forms is actually straightforward and well-known. It reads as follows:
	
	\begin{thmA}\label{thmA}
		Assume that $(\Sigma^2,g)$ with Gaussian curvature $K$ is isometrically immersed into $\mathbb{Q}^3(c)$ with constant mean curvature $H$. Then $c+H^2-K$ is non-negative, vanishes either everywhere or on a discrete set of points, and on the complement of its zero set the Riemannian metric $\sqrt{c+H^2-K}g$ is flat. 
		
		Conversely, given $(\Sigma^2,g)$, suppose that there exists a pair of constants $(c,H)$ such that $c+H^2-K$ is positive everywhere and the Riemannian metric $\sqrt{c+H^2-K}g$ is flat. Then $(\Sigma^2,g)$ admits, locally, a constant mean curvature $H$ isometric immersion into $\mathbb{Q}^3(c)$.
	\end{thmA}	
	In the converse statement, we remark that the universal cover of $(\Sigma^2,g)$ always admits a constant mean curvature $H$ isometric immersion into $\mathbb{Q}^{3}(c)$.
		
	We want to sketch the proof of the above result, emphasising the role played by the traceless part of the second fundamental form,
	\begin{equation*}
		\mathring{A}=A-Hg,
	\end{equation*}
	which is conformally invariant, and by a Simons-type identity for the Laplacian of the squared-norm of $\mathring{A}$ on $(\Sigma^2,g)$, see \eqref{eqSimons} below.	
	\begin{proof}
		$(\Rightarrow)$ The Gauss equation can be rewritten as
		\begin{equation} \label{eqGaussrewritten}
			K = c + H^2 - \frac{1}{2}|\mathring{A}|^2.
		\end{equation}
		Thus, $c+H^2-K=|\mathring{A}|^2/2$ is non-negative on $\Sigma^2$.
		
		The Codazzi equation implies that, in local isothermal coordinates $z=x+iy$ compatible with the Riemann surface structure of $(\Sigma^2,g)$, the complex function
		\begin{equation}\label{eqholomorphic}
			\phi(z) = 2\mathring{A}\left(\frac{\partial}{\partial z},\frac{\partial}{\partial z}\right) = \frac{1}{2}\left(\mathring{A}\left(\frac{\partial}{\partial x},\frac{\partial}{\partial x}\right)- \mathring{A}\left(\frac{\partial}{\partial y},\frac{\partial}{\partial y}\right)\right) - i\mathring{A}\left(\frac{\partial}{\partial x},\frac{\partial}{\partial y}\right)	
		\end{equation}
	is holomorphic. Since $\phi(z)=0$ if, and only if, the traceless tensor $\mathring{A}$ vanishes at $z$, we conclude that the function $c+H^2-K=|\mathring{A}|^2/2$ either vanishes everywhere or vanishes on a discrete set of points of $\Sigma^2$. 
	
		Finally, combining Gauss and Codazzi equations, a straightforward computation of the Laplacian of $|\mathring{A}|^2$ in $(\Sigma^2,g)$ leads to the identity
		\begin{equation} \label{eqSimons}
			\frac{1}{2}\Delta |\mathring{A}|^2 = |\nabla \mathring{A}|^2 + 2K|\mathring{A}|^2.
		\end{equation}
		Since, in dimension $n=2$, any traceless tensor $\mathring{A}$ satisfies 
		\begin{equation*}	
			|\nabla \mathring{A}|^2 = 2|\nabla|\mathring{A}||^2
		\end{equation*} 
at points where $|\mathring{A}|\neq 0$, equation \eqref{eqSimons} can be rewritten as 
		\begin{equation*}
			\Delta \log(|\mathring{A}|) = 2K
		\end{equation*}
	on the set of points where $\mathring{|A|}\neq 0$. 
		
		Recalling the equation that relates the Gaussian curvature of conformal metrics $h=e^{2u}g$, namely
	\begin{equation} \label{eqconformal1}
		-\Delta_g u + K_g = K_h e^{2u}, 
	\end{equation}		
	 we conclude that the conformal factor $e^{2u}=|\mathring{A}|$  turns $g$ into a flat metric. The same can be said about the conformal factor $\sqrt{c+H^2-K}=|\mathring{A}|/\sqrt{2}$.
	 
	 $(\Leftarrow)$ Consider local isothermal coordinates $z=x+iy$ on a disc in $(\Sigma^2,g)$, compatible with the underlying Riemann surface structure, so that $g=e^{2u}(dx^2+dy^2)$ for some smooth function $u$. By hypothesis, $\sqrt{c+H^2-K}g$ is a flat Riemannian metric. By equation \eqref{eqconformal1}, the assumption is therefore equivalent to the equation
	 \begin{equation*}	
		\Delta (\log(c+H^2-K)e^{4u}) = 0,
	 \end{equation*} 
where $\Delta=\partial_x^2+\partial_y^2$. By Lemma \ref{lemlogharmonic} (see the Appendix), there exists a holomorphic function $\phi(z)$ on this coordinate neighbourhood such that
	 \begin{equation*}
	 	|\phi|^2 = (c+H^2-K)e^{4u}.
	 \end{equation*}
	 
	 Next, define a two-tensor $A$ so that, on this coordinate neighbourhood, 
	 \begin{align*}
	 	A(\partial_x,\partial_x) & = Re(\phi) + He^{2u}, \\
	 	A(\partial_y,\partial_y) & = -Re(\phi) + He^{2u}, \\
	 	A(\partial_x,\partial_y) & = -Im(\phi)  = A(\partial_y,\partial_x).
	 \end{align*}
	 Then, $A$ defines a symmetric two-tensor, with average trace $H$, whose traceless part has squared-norm $2|\phi|^2e^{-4u}$ with respect to the metric $g=e^{2u}(dx^2+dy^2)$. (We may use $\{e^{-u}\partial_x,e^{-u}\partial_y\}$ as $g$-orthonormal frame to perform these computations).
	 
	 A long but straightforward computation then shows that $(g,A)$ satisfy both the Gauss equation (rewritten as in \eqref{eqGaussrewritten}) and the Codazzi equation for a constant mean curvature $H$ isometric immersion into $\mathbb{Q}^3(c)$. By the Fundamental Theorem of Isometric Immersions, this simply connected coordinate neighbourhood of $(\Sigma^2,g)$ isometrically immerses into $\mathbb{Q}^3(c)$ with constant mean curvature $H$.
	\end{proof}
	
	\begin{rmk}\label{examLawson}
		Lawson gave an example of Riemannian metric $g$ on $\mathbb{R}^2$ such that its Gaussian curvature is negative everywhere except at the origin, such that the Riemannian metric $\sqrt{-K}g$ is flat where it is defined, and yet $(\mathbb{R}^2,g)$ does not isometrically immerse as a minimal surface into $\mathbb{R}^3$ (see \cite{Law}, Remark 12.1). Thus, in a sense, Theorem A is sharp. Nevertheless, it is not difficult to see what fails in his example. As soon as one understand the special structure of the zeroes of $K$ - namely, that they are the zeroes of $|\phi|^2$ for some holomorphic function $\phi$ in local isothermal coordinates - one can figure out the right extra necessary condition that rules out Lawson's examples. This observation has been explored in details by Andrei Moroianu and Sergiu Moroianu \cite{MorMor}, who introduced recently the very apt concept of \textit{Ricci surfaces}. We will discuss this notion further in Section \ref{subsecRicci}.	
	\end{rmk}	
	
\subsection{The Lawson correspondence} There is an intriguing observation regarding the converse statement of Theorem A. If different pairs of constants $(c,H)$ and $(\tilde{c},\tilde{H})$ satisfy $c+ H^2=\tilde{c}+\tilde{H}^2$, then the sufficient condition that needs to be verified is satisfied by the pair $(c,H)$ if, and only if, it is satisfied by the pair $(\tilde{c},\tilde{H})$. This means that every simply connected CMC surface in a space form is isometric to a CMC ``cousin" surface with different constant mean curvature in a different space form.
	
	This correspondence is known as the \textit{Lawson correspondence}. It has motivated many interesting developments, for instance, in the theory of CMC $H=1$ surfaces in the hyperbolic space (also known as Bryant surfaces after \cite{Bry}), which are simply the cousins to minimal surfaces in $\mathbb{R}^3$.
	
	\begin{exam}	
		Complete minimal surfaces of revolution in the Euclidean space $\mathbb{R}^3$ have been classified long ago: they are the \textit{catenoids}. Up to isometries and dilations of $\mathbb{R}^3$, there exists only one catenoid. Robert Bryant showed that catenoids are in correspondence with complete CMC $H= 1$ surfaces of revolution in the hyperbolic space $\mathbb{H}^3=\mathbb{Q}^3(-1)$, called the \textit{catenoid cousins} \cite{Bry}. Some catenoid cousins are embedded, while others are immersed but not embedded. It is interesting to notice that dilated Euclidean catenoids correspond to very different hyperbolic catenoid cousins.
	\end{exam}

\subsection{Higher dimensions} Let $(\Sigma^n,g)$ be a connected, oriented Riemannian manifold without boundary, not necessarily complete, that is isometrically immersed into the space form $\mathbb{Q}^{n+1}(c)$. In dimensions $n\geq 3$, while the Codazzi equation for the second fundamental form $A$ is the same, the Gauss equation for the Riemman curvature tensor of $g$ and its traces read as
	\begin{align}
		Rm & = \frac{1}{2}c\, g\odot g + \frac{1}{2}A\odot A, \label{eqGauss1}\\
		Ric & = c(n-1)g + nHA - A\circ A, \label{eqGauss2} \\
		R & = cn(n-1) + n^2H^2 - |A|^2. \nonumber	
	\end{align} 
	In \eqref{eqGauss1}, $\odot$ denotes the Kulkarni-Nomizu product of symmetric two-tensors, while in \eqref{eqGauss2}, $A\circ A$ denotes the symmetric two-tensor $(A\circ A)(X,Y)= \sum_{i=1}^n A(X,e_i)A(e_i,Y)$ for every $X$, $Y\in \mathcal{X}(\Sigma)$, where $\{e_i\}$ is any local $g$-orthonormal frame.
	
	It will be convenient to have the following notation at hand.  Let $\hat{A}$ be the (fiberwise) endomorphism of the tangent bundle defined by 
	\begin{equation*}
		g(\hat{A}(X),Y)=A(X,Y) \quad \text{for every} \quad X,Y\in \mathcal{X}(\Sigma).
	\end{equation*}
	Then $\hat{A}$ is symmetric with respect to $g$ and
	\begin{equation*}
		(A\circ A)(X,Y)=g(\hat{A}(X),\hat{A}(Y)) \quad \text{for every} \quad X,Y\in \mathcal{X}(\Sigma).
	\end{equation*}
	Also, by construction, 
	\begin{equation*}
		(\nabla_X A)(Y,Z)=g((\nabla_X\hat{A})(Y),Z) \quad \text{for every} \quad X,Y,Z\in \mathcal{X}(\Sigma).
	\end{equation*}
		
	Rewriting the equation \eqref{eqGauss2} in terms of the traceless part $\mathring{A} = A - Hg$ of $\mathring{A}$, we obtain
	\begin{equation*}
		Ric = (n-1)(c+H^2)g + (n-2)H^2\mathring{A} - \mathring{A}\circ \mathring{A}.
	\end{equation*}
	Thus, in contrast with the case $n=2$ (where $Ric=Kg$), except when $H$ vanishes identically, it is not clear whether the symmetric two-tensor $(n-1)(c+H^2)g -Ric$ is non-negative or not. (At least, it does not needs to be so from a purely \textit{algebraic} point of view, even when assuming $H$ is constant). 
	
	In order to continue the analysis, we therefore restrict our attention to the case $H=0$. Chern and Osserman \cite{CheOss}  analysed the case $c=0$ in detail, but the final statement, for arbitrary values of $c$, is precisely the theorem proven by Dajczer and do Carmo in \cite{dCDaj}. We state it as follows:

	\begin{thmB}
		Suppose that $(\Sigma^n,g)$ is isometrically immersed into $\mathbb{Q}^{n+1}(c)$ as a minimal hypersurface, and denote by $A$ its second fundamental form. Then the symmetric two-tensor $\overline{g} = c(n-1)g - Ric_g$ is non-negative, and wherever it is positive definite, it is a Riemannian metric such that
		\begin{itemize}
			\item[$i)$] $\overline{g}(\overline{\nabla}_X Y,Z) = g(\nabla_X (\hat{A}(Y)),\hat{A}(Z))$ for every $X$, $Y$, $Z\in \mathcal{X}(\Sigma)$,
			\item[$ii)$] $Rm_{\overline{g}} = \frac{1}{2}\overline{g}\odot \overline{g} + \frac{c}{2} A\odot A$.
		\end{itemize}		
		
		Conversely, assume that $(\Sigma^n,g)$ is such that $\overline{g}=c(n-1)g-Ric_g$ is a Riemannian metric on $\Sigma^n$, and that there exists a smooth symmetric two-tensor $A$ on $(\Sigma^n,g)$ such that $\overline{g}=A\circ A$ satisfies properties $i)$ and $ii)$ above. Then $(\Sigma^n,g)$ admits, locally, a minimal isometric immersion into $\mathbb{Q}^{n+1}(c)$.
	\end{thmB}	
	
	When $c=0$, the condition $ii)$ above holds if, and only if, $\overline{g}$ is a metric of constant sectional curvature $1$. In fact, in this case $\overline{g}$ is nothing but the pulled-back metric from the sphere $\mathbb{S}^n(1)\subset \mathbb{R}^{n+1}$ by the Gauss map $N$. 
	
	The original proof in \cite{dCDaj} was formulated in the language of Cartan's \textit{rep\`ere mobile} method. We sketch a proof below using common tensor language, as in \cite{DajToj}, Section 3.6.

\begin{proof}
	$(\Rightarrow)$ The second fundamental form $A$ has average trace $H=0$ and, by the Gauss equation \eqref{eqGauss2},
	\begin{equation*}
		(n-1)cg - Ric_g = A\circ A 
	\end{equation*}
	is a non-negative symmetric tensor, because $(A\circ A)(X,X)=|\hat{A}(X)|^2$ for all tangent vector fields $X$ on $(\Sigma^n,g)$. 
	
	Consider the set where $\overline{g} = (n-1)cg - Ric_g=A\circ A$ is positive definite. It is the same set as the set of points where $\hat{A}$ is invertible. By a straightforward computation using the Koszul equation, the Levi-Civita connection $\overline{\nabla}$ of $\overline{g}$ satisfies
	\begin{align} \label{eqkoszul}
		2\overline{g}(\overline{\nabla}_X Y,Z) =&\,  2g(\nabla_X (\hat{A}(Y)),\hat{A}(Z))   \\
		& + \langle (\nabla_X \hat{A})(Z) - (\nabla_Z \hat{A})(X), \hat{A}(Y) \rangle \nonumber \\ 
		& + \langle (\nabla_Y A)(Z) - (\nabla_Z \hat{A})(Y), \hat{A}(X) \rangle \nonumber \\
		& - \langle (\nabla_X \hat{A})(Y)-(\nabla_Y \hat{A})(X),\hat{A}(Z)\rangle \nonumber .
	\end{align}
	Thus, the Codazzi equation implies $i)$.  
	
	Another straightforward computation using $i)$ and the definition of the Riemann curvature tensor $\overline{Rm}$ of $\overline{g}$ gives
	\begin{equation}\label{eqcurvaturetensor}
		\overline{Rm}(X,Y,X,Y) = Rm(X,Y,\hat{A}(X),\hat{A}(Y)) \,\, \text{for every} \,\, X,Y \in \mathcal{X}(\Sigma).
	\end{equation} 
	By the Gauss equation \eqref{eqGauss1},
	\begin{align*}\label{eqcurvaturetensor}
		\overline{Rm}(X,Y,X,Y) & = c (g(X,\hat{A}(X))g(Y,\hat{A}(Y)) - g(X,\hat{A}(Y))g(Y,\hat{A}(X)) ) \\
		& \quad + A(X,\hat{A}(X))A(Y,\hat{A}(Y))-A(X,\hat{A}(Y))A(Y,\hat{A}(X)). \\
		& = c ( A(X,X)A(X,Y) - A(X,Y)^2) \\
		& \quad + \overline{g}(X,X)\overline{g}(Y,Y) - \overline{g}(X,Y)^2.
	\end{align*} 
	Since $X$ and $Y$ are arbitrary (and a four-tensor $T$ with the same symmetries of the Riemann curvature tensor is uniquely determined by the values $T(X,Y,X,Y)$), $ii)$ follows.
	
	$(\Leftarrow)$ By definition of $A$, $|\hat{A}(X)|^2=(A\circ A)(X,X)=\overline{g}(X,X)$ for every tangent vector $X$. Since, by assumption, $\overline{g}$ is positive definite, $\hat{A}$ is invertible. 
	
	Using the symmetry of the Levi-Civita connection $\overline{\nabla}$ of $\overline{g}$, equation \eqref{eqkoszul} together with $i)$ implies that, for every tangent vectors $X$, $Y$ and $Z$ on $\Sigma^n$,
	\begin{align*}
		0 & = \overline{g}(\overline{\nabla}_X Y - \overline{\nabla}_Y X-[X,Y],\hat{A}^{-1}(Z)) \\
		& = g(\nabla_X(\hat{A}(Y)) - \nabla_Y(\hat{A}(X)) - \hat{A}([X,Y]),Z)\\
		& = g( (\nabla_X \hat{A})(Y) - (\nabla_Y \hat{A})(X),Z) \\
		& = (\nabla_X A)(Y,Z) - (\nabla_Y A)(X,Z).
	\end{align*} 
	Hence, $A$ satisfies the Codazzi equation on $(\Sigma^n,g)$.
	
	As seen above, equation \eqref{eqcurvaturetensor} is a consequence of the definition of $\overline{g}$ and $A$ and together with $i)$. Since $\hat{A}$ is invertible, \eqref{eqcurvaturetensor} and $ii)$ implies that
	\begin{align*}
		Rm(X,Y,X,Y) = &  \overline{Rm}(X,Y,\hat{A}^{-1}(X),\hat{A}^{-1}(Y)) \\ 
	 	= &  \overline{g}(X,\hat{A}^{-1}(X))\overline{g}(Y,\hat{A}^{-1}(Y)) - \overline{g}(X,\hat{A}^{-1}(Y))\overline{g}(Y,\hat{A}^{-1}(X)) \\
		& + c A(X,\hat{A}^{-1}(X))A(Y,\hat{A}^{-1}(Y)) \\
		& - c A(X,\hat{A}^{-1}(Y))A(Y,\hat{A}^{-1}(X)) \\
		= & g(\hat{A}(X),X)g(\hat{A}(Y),Y)-g(\hat{A}(X),Y)g(\hat{A}(Y),X) \\
		& + cg(X,X)g(Y,Y)-cg(X,Y)^2
	\end{align*}
	for all $X$, $Y\in \mathcal{X}(\Sigma)$. In other words, $A$ satisfy the Gauss equation \eqref{eqGauss1} for an isometric immersion of $(\Sigma^n,g)$ into $\mathbb{Q}^{n+1}(c)$. 
	
	Therefore the Fundamental Theorem of Isometric Immersions applies for the pair $(g,A)$ on any simply connected neighbourhood in $\Sigma$, and provides local isometric immersions of $(\Sigma^n,g)$ into $\mathbb{Q}^{n+1}(c)$ with second fundamental form $A$. \textit{A posteriori}, since $Ric_g = c(n-1)g - A\circ A$ (by assumption) as well as $Ric_g =c(n-1)g +nHA -A\circ A$ (by \eqref{eqGauss2}), and since $\hat{A}$ is invertible, we must have $H=0$. In other words, these local isometric immersions are minimal.
\end{proof}
	
	\begin{rmk}
		For minimal immersed hypersurfaces in $\mathbb{Q}^{n+1}(c)$, $n\geq 3$, the structure of the set where $c(n-1)g - Ric_g = A\circ A$ ceases to be positive definite has a more complicated structure. In three dimensions, $A\circ A$ is not positive definite at a point $p$ if, and only if, either $\hat{A}$ has a one-dimensional kernel or vanishes at $p$. In higher dimensions, the situation is of course much wilder. (See, however, the discussions in \cite{CheOss} about the specific case $c=0$). We think that it could be interesting to understand this behaviour somehow, specially in three-dimensions, because this understanding could lead to the formulation of a meaningful higher dimensional generalisation of the concept of Ricci surfaces in the sense of \cite{MorMor}, \textit{cf}. Section \ref{subsecRicci}.
	\end{rmk}
	
\section{The free boundary condition} 	

	After this overview, let us consider isometric immersions of a connected, oriented Riemannian manifold $(\Sigma^n,g)$ with boundary $\partial \Sigma\neq \emptyset$, not necessarily complete. Our aim is to find analogues of Theorems A and B under geometrically reasonable boundary conditions on the pair $(g,A)$, added to the Gauss and Codazzi equations.
	
	Before we formulate these boundary conditions, let us fix some notation once and for all. The unit outward-pointing conormal of $\partial \Sigma$ inside $\Sigma^n$ will be denoted by $\nu$. If $\partial_i \Sigma$ is a boundary component of $\Sigma^n$, we denote by $B_i=g(\nabla_{-}\nu,-)$ its second fundamental form in $(\Sigma^n,g)$.
	
\subsection{Umbilical hypersurfaces and capillary boundary conditions} \label{subsecumbilical} Inside any space form $\mathbb{Q}^{n+1}(c)$ there are distinguished hypersurfaces called \textit{umbilical}, which are defined by the property that their second fundamental form is a multiple of the metric at every point. Connected umbilical hypersurfaces in $\mathbb{Q}^{n+1}(c)$ have therefore constant mean curvature, and they are easily classified. In fact, connected umbilical hypersurfaces are just the standard embeddings of space forms $\mathbb{Q}^{n}(b)$ in $\mathbb{Q}^{n+1}(c)$, which exist only when $b\geq c$. (See Proposition 1.20 in \cite{DajToj}).

	Denote by $\overline{N}$ the unit normal to an umbilical hypersurface $\mathbb{Q}^{n}(b)$ of $\mathbb{Q}^{n+1}(c)$, chosen in such way that its mean curvature $\overline{H}$ is non-negative. The relation between the ambient sectional curvature $c$, the intrinsic sectional curvature $b$ and the mean curvature $\overline{H}$ of the umbilical immersion of $\mathbb{Q}^n(b)$ into $\mathbb{Q}^{n+1}(c)$ is 
	\begin{equation*}
		b=c+\overline{H}^2.
	\end{equation*}
(Notice that $\overline{H}=0$ if, and only if, $b=c$, in which case the second fundamental form vanishes identically, by virtue of the Gauss equation \eqref{eqGauss1}. The choice of $\overline{N}$ is ambiguous in this case, and can be determined only by a further choice of orientation).

	The geometric situation that interests us is when a hypersurface $\Sigma^n$ of $\mathbb{Q}^{n+1}(c)$ intersects an umbilical hypersurface $\mathbb{Q}^n(b_i)$ transversally along a boundary component $\partial_i \Sigma$. Moreover, we want to suppose that the angle between $\nu$ and the unit normal $\overline{N}$ to $\mathbb{Q}^n(b_i)$ is constant along $\partial_i \Sigma$. This boundary condition is known as the \textit{capillary condition}. In the particular case where $\overline{N}=\nu$ or $\overline{N}=-\nu$ on $\partial_i \Sigma$, that is, when $\Sigma^n$ meets $\mathbb{Q}^n(b_i)$ orthogonally along $\partial_i\Sigma$, we say that $\Sigma^n$ is \textit{free boundary} with respect to $\mathbb{Q}^n(b_i)$ along $\partial_i\Sigma$.  
	
	Since $\partial_i\Sigma$ is contained in $\mathbb{Q}^{n}(b_i)$, there exists a unique unit vector field $\overline{\nu}$ along $\partial_i \Sigma$, tangent to $\mathbb{Q}^n(b_i)$ but normal to $\partial_i \Sigma$, in such way that $\overline{N}$ and $\overline{\nu}$ (in that order) form together another orthonormal normal frame along $\partial_i\Sigma$ that induces the same orientation on the normal bundle of $\partial_i \Sigma$ as $N$ and $\nu$ (in that order). Assuming the capillary condition, there exists moreover a constant $\alpha_i\in (0,\pi)\cup (\pi,2\pi)$ such that
	\begin{align*}
		\overline{N} & = \cos(\alpha_i)N + \sin(\alpha_i)\nu,  \\
		\overline{\nu} & = -\sin(\alpha_i)N + \cos(\alpha_i) \nu.
	\end{align*} 
		
		Equations \eqref{eqfund} below will be key to the results we are searching for.	
	\begin{lemm} \label{lemma}
	Under the capillary conditions, for every tangent vector fields $X$ and $Y$ on $\partial_i \Sigma$, we have
		\begin{itemize}
			\item[$i)$] $A(X,\nu)=0$,
			\item[$ii)$] $\overline{H}g(X,Y)= \cos(\alpha_i)A(X,Y) + \sin(\alpha_i)B_i(X,Y)$.
		\end{itemize}
		In particular, in the free boundary case,
		\begin{equation} \label{eqfund}
			A(X,\nu) = 0 \quad \text{and} \quad B_i = \pm \sqrt{b_i-c}g,
		\end{equation}
		where the sign of the trace of $B_i$ depends on whether $\nu=\overline{N}$ or $\nu=-\overline{N}$. 
	\end{lemm}  
	\begin{proof}
		Let $X$ be a vector field tangent to $\partial_i \Sigma$. The two vector fields $X$ and $\overline{\nu}$ are tangent to $\mathbb{Q}^{n}(b_i)$ and orthogonal between themselves. Since $\mathbb{Q}^{n}(b_i)$ is an umbilical hypersurface,
		\begin{equation*}
			\langle D_X\overline{N},\overline{\nu}\rangle = \overline{H}\langle X,\overline{\nu}\rangle = 0.
		\end{equation*}  		
		On the other hand, substituting the expressions of $\overline{N}$ and $\overline{\nu}$ in terms of the orthonormal frame $N$ and $\nu$ and using that $\alpha_i$ is constant,
		\begin{align*}
			\langle D_X\overline{N},\overline{\nu}\rangle & = \cos(\alpha_i)^2\langle D_X N,\nu \rangle - \sin(\alpha_i)^2\langle D_X \nu, N\rangle \\ 
			& = \langle D_X N, \nu\rangle = A(X,\nu),
		\end{align*}
		by definition of $A$. Item $i)$ follows.	
			
		Item $ii)$ is also a consequence of a direct computation. For every vector fields $X$ and $Y$ tangent to $\partial_i\Sigma$,
		\begin{align*}
			\langle D_X \overline{N},Y\rangle & = \cos(\alpha_i)\langle D_X N,Y\rangle +\sin(\alpha_i)\langle D_X \nu,Y\rangle \\
			& = \cos(\alpha_i)A(X,Y) + \sin(\alpha_i)g(\nabla_X \nu,Y).
		\end{align*}
		In particular, when $\alpha_i=\pi/2$ or $\alpha_i=3\pi/2$, we have
		\begin{align*}
			B_i(X,Y) = g(\nabla_X \nu,Y)= \pm \langle D_X \overline{N},Y\rangle = \pm \overline{H}g(X,Y),
		\end{align*}
		because $\mathbb{Q}^n(b_i)$ is umbilical in $\mathbb{Q}^{n+1}(c)$. Since $b_i=c+\overline{H}^2$ and we have chosen orientations in such way that $\overline{H}\geq 0$, the result follows.
	\end{proof}		
	
	In dimension $n=2$, we can proceed our analysis of the free boundary condition a bit further, and derive properties of the squared-norm of $\mathring{A}$ along the boundary that will be crucial for our reasoning.
	
	Denote by $T$ a unit tangent vector field along the boundary component $\partial_i\Sigma$, so that $T$ and $\nu$ form an orthonormal frame along $\partial_i \Sigma$. Since $\mathring{A}$ has zero trace with respect to $g$ by definition, 
	\begin{equation}\label{eqmathringtrace}
		\mathring{A}(T,T)+\mathring{A}(\nu,\nu)=0.
	\end{equation}
	Moreover, by Lemma \ref{lemma},
	\begin{equation}\label{eqmathringTnu}
		\mathring{A}(T,\nu)=A(T,\nu)+Hg(T,\nu)=0.
	\end{equation}
	Therefore
	\begin{equation}\label{eqnormofA}
		|\mathring{A}|^2 = \mathring{A}(T,T)^2+\mathring{A}(\nu,\nu)^2+2\mathring{A}(T,\nu)^2 = 2\mathring{A}(T,T)^2.
	\end{equation}
	
	Since $\nabla_T T$ is ortogonal to $T$, $\nabla_T \nu$ is orthogonal to $\nu$, and these vector fields satisfy
	\begin{equation*}
		g(\nabla_T T,\nu) = -g(T,\nabla_T \nu)=-B_i(T,T) =\mp\sqrt{b_i-c}
	\end{equation*}
	because of Lemma \ref{lemma}, we have
	\begin{equation}\label{eqnablas}
		\nabla_T \nu =  \pm\sqrt{b_i-c}\,T \quad \text{and} \quad \nabla_T T = \mp  \sqrt{b_i-c}\,\nu.
	\end{equation}

	\begin{lemm}\label{lemmaBen}
		Under the free boundary condition along $\partial_i \Sigma \subset \mathbb{Q}^2(b_i)$,
		\begin{equation}\label{eq|A|bdry}
			-\frac{\partial}{\partial \nu}|\mathring{A}|^2 = \pm 4\sqrt{b_i-c}|\mathring{A}|^{2} \quad \text{on} \quad \partial_i \Sigma,
		\end{equation}
		where the sign depends on whether $\nu=\overline{N}$ or $\nu=-\overline{N}$.
	\end{lemm}
	
	\begin{proof}
		Using \eqref{eqmathringtrace} and \eqref{eqmathringTnu},
		\begin{align*}
			\frac{1}{2}\frac{\partial}{\partial \nu}|\mathring{A}|^2 = & \, g(\nabla_\nu \mathring{A},\mathring{A}) \\
			= &\, (\nabla_\nu \mathring{A})(T,T)\mathring{A}(T,T) + (\nabla_\nu \mathring{A})(\nu,\nu)\mathring{A}(\nu,\nu) \\
			  &\, + 2(\nabla_\nu \mathring{A})(T,\nu)\mathring{A}(T,\nu) \\
			= &\,  2(\nabla_\nu \mathring{A})(T,T)\mathring{A}(T,T) = 2 (\nabla_T \mathring{A})(\nu,T)\mathring{A}(T,T),
		\end{align*}
	by the Codazzi equation. (Since $H$ is constant, $\nabla A = \nabla \mathring{A})$.
	
	By definition,
	\begin{equation*}
	  (\nabla_T \mathring{A})(\nu,T) = T \mathring{A}(\nu,T) - \mathring{A}(\nabla_T\nu,T)-\mathring{A}(\nu,\nabla_T T).	
	\end{equation*}	
	Substituting \eqref{eqmathringTnu}, \eqref{eqnormofA}, \eqref{eqmathringtrace} and \eqref{eqnablas} successively, we obtain
	\begin{align*}
		-\frac{\partial}{\partial \nu}|\mathring{A}|^2 = & \pm 4\sqrt{b_i-c}\left(\mathring{A}(T,T) - \mathring{A}(\nu,\nu) \right) \mathring{A}(T,T) \\
	= & \pm 8\sqrt{b_i-c}\mathring{A}(T,T)^2 = \pm 4\sqrt{b_i-c}|\mathring{A}|^{2}.
	\end{align*}
	\end{proof}

\subsection{Dimension two} Let $(\Sigma^2,g)$ be a connected, oriented Riemannian surface with non-empty boundary $\partial \Sigma$, not necessarily complete. Our generalisation of Theorem A under the free boundary condition reads as follows: 
		
	\begin{thmC}
		Assume that $(\Sigma^2,g)$ with Gaussian curvature $K$ is isometrically immersed into $\mathbb{Q}^3(c)$ with constant mean curvature $H$ and in such way that it intersects a collection of umbilical $\mathbb{Q}^{2}(b_i)$ orthogonally along the boundary components $\partial_i \Sigma$. Then $c+H^2-K$ is non-negative, vanishes on a discrete set of points, and on the complement of its zero set the Riemannian metric $\sqrt{c+H^2-K}g$ is flat and such that each $\partial_i\Sigma$ has zero geodesic curvature with respect to it.
		
		Conversely, let $(\Sigma^2,g)$ be such that its boundary components $\partial_i \Sigma$ have constant geodesic curvature $k_i=\pm\sqrt{b_i-c}$ for some constants $b_i\geq c$, and assume  that there exists a constant $H$ such that the function $c+H^2-K$ is positive everywhere, that the Riemannian metric $\sqrt{c+H^2-K}g$ is flat and that the boundary components are geodesics with respect to that metric. Then $(\Sigma^2,g)$ admits, locally, constant mean curvature $H$ isometric immersions into $\mathbb{Q}^3(c)$ that meet a collection of umbilical $\mathbb{Q}^{2}(b_i)$ orthogonally on points of the boundary component $\partial_i\Sigma$.
	\end{thmC}
	
	\begin{proof}
		$(\Rightarrow)$ In the proof of Theorem A, we have already analysed the neighbourhoods of points in $\Sigma\setminus \partial \Sigma$. Consider isothermal coordinates $z=x+iy$ in a neighbourhood of a boundary point in $\partial_i\Sigma$, compatible with the underlying Riemann surface structure of $(\Sigma^2,g)$, in such way that the points with $y>0$ are interior points and $\{y=0\}$ parametrises boundary points. 
		
	In this coordinate system, at each point, $\nu$ is parallel to $-\partial_y$ and $\partial_x$ is tangent to $\partial_i \Sigma$ on $y=0$. By Lemma \ref{lemma}, $A(\partial_x,\partial_y)=0$ on $y=0$. Hence, the holomorphic function $\phi$ defined in \eqref{eqholomorphic} is such that its imaginary part vanishes at $y=0$. By Schwartz reflection, $\phi$ extends as a holomorphic function in the reflected domain across $\{y=0\}$. In particular, $\phi$ vanishes at isolated boundary points, or vanishes identically. Since $\phi$ vanishes precisely at points where $\mathring{A}=0$, the same can be said about the zero set of $\sqrt{c+H^2-K}=|\mathring{A}|^2/2$.
	
	At boundary points where $\mathring{A}\neq 0$, the conformal factor $e^{2u}=|\mathring{A}|$, which turns $\overline{g}=e^{2u}g$ into a flat metric as in the proof of Theorem A, also satisfies
	\begin{equation*}
		\frac{\partial}{\partial \nu} \log |\mathring{A}| \pm 2\sqrt{b_i-c^2} = 0,
	\end{equation*} 
	by virtue of Lemma \ref{lemmaBen}. Recalling the equation that relates the geodesic curvature of $\partial_i \Sigma$ in $\Sigma^2$ (with respect to $\nu$) for  conformal metrics $h=e^{2u}g$, namely
	\begin{equation} \label{eqconformal2}
		\frac{\partial u}{\partial \nu_g} + k_g = k_h e^{u}, 
	\end{equation}		 
	and since the geodesic curvature of $\partial_i \Sigma$ in $(\Sigma^2,g)$ is precisely $\pm\sqrt{b_i-c}$ by virtue of Lemma \ref{lemma}, the geodesic curvature of $\partial_i \Sigma$ with respect to the conformal metric $|\mathring{A}|g$ is zero. The same can be said about $\partial_i \Sigma$ with respect to the conformal metric $\overline{g}=\sqrt{c+H^2-K}g=(|\mathring{A}|/\sqrt{2})g$.
	
	$(\Leftarrow)$ For simply connected neighbourhoods of interior points of $\Sigma^2$, we argue as in the proof of Theorem A. Consider then local isothermal coordinates $z=x+iy$ on a half-disc on $(\Sigma^2,g)$, compatible with the underlying Riemann surface structure, so that interior points correspond to $y>0$ and boundary points of $\partial_i \Sigma$ correspond to $y=0$. We have $g=e^{2u}(dx^2+dy^2)$ for a smooth function $u$.
	
	 By assumption, $(\Sigma^2,\sqrt{c+H^2-K}g)$ has zero Gaussian curvature and its boundary has zero geodesic curvature. Equivalently, in view of \eqref{eqconformal1} and \eqref{eqconformal2},
	 \begin{equation*}	
		\Delta \log((c+H^2-K)e^{4u}) = 0 \quad \text{and} \quad \frac{\partial}{\partial y} \log((c+H^2-K)e^{4u}) = 0,
	 \end{equation*} 
	where $\Delta=\partial_x^2+\partial_y^2$. By Lemma \ref{lemlogharmonic} (see Appendix), there exists a holomorphic function $\phi(z)$ on this coordinate neighbourhood that has zero imaginary part on $\{y=0\}$ and is such that that
	 \begin{equation*}
	 	|\phi|^2 = (c+H^2-K)e^{4u}.
	 \end{equation*}

	As in the proof of Theorem A, let $A$ be the two-tensor defined in this coordinates system as follows:
	 \begin{align*}
	 	A(\partial_x,\partial_x) & = Re(\phi) + He^{2u}, \\
	 	A(\partial_y,\partial_y) & = -Re(\phi) + He^{2u}, \\
	 	A(\partial_x,\partial_y) & = -Im(\phi)  = A(\partial_y,\partial_x).
	 \end{align*}
	 Then, $A$ is a symmetric two-tensor, with average trace $H$, whose traceless part has squared-norm $2|\phi|^2e^{-4u}$ with respect to the metric $g=e^{2u}(dx^2+dy^2)$. Moreover, since $\phi$ has zero imaginary part on $\{y=0\}$, $A(\partial_x,\partial_y)=0$ on this set.
	 
	 As in the proof of Theorem $C$, the pair $(g,A)$ satisfies the Gauss and Codazzi equations for a constant mean curvature $H$ isometric immersion into $\mathbb{Q}^3(c)$ on this neighbourhood. By the Fundamental Theorem of Isometric Immersions, this simply connected region isometrically immerses into $\mathbb{Q}^3(c)$ with second fundamental form $A$ and constant mean curvature $H$. Denote by $N$ the unit normal with $A=\langle D_{-}N,-\rangle$, and by $T$ a unit vector field tangent to $\partial_i\Sigma$ in points of this region.
	 
	 It remains to compute the image of the points of $\partial_i \Sigma$ under this immersion. First of all, notice that $A(T,\nu)=0$ along points of $\partial \Sigma$ in this neighbourhood. In fact, $\nu=-e^{-u}\partial_y$ and $T=\pm e^{-u}\partial_x$ at each boundary point, so that the vanishing of $A(T,\nu)$ follows from the definition of $A$ and the vanishing of the imaginary part of $\phi$. 
	 
	 It follows that $\nu$ is a parallel section of the normal bundle of $\partial_i \Sigma$ in $\mathbb{Q}^{3}(c)$. In fact, $D_T \nu$ has no normal component because $D_T \nu$ is orthogonal to $\nu$, and $\langle D_T \nu,N\rangle = -\langle D_T N,\nu\rangle=-A(\nu,T)=0$.
	 
	 Moreover, $\langle D_T \nu,T\rangle = g(\nabla_T\nu,T)$ is the geodesic curvature of $\partial_i \Sigma$ in $(\Sigma^2,g)$, which is constant and equal to $\pm \sqrt{b_i-c}$ by assumption.
	 
	Hence, standard results imply that the boundary points in question lie inside an umbilical $\mathbb{Q}^{2}(x_i)$ in $\mathbb{Q}^{n+1}(c)$, whose Gaussian curvature is $x_i = c + (\pm\sqrt{b_i-c})^2 = b_i$, and moreover $\nu$ is normal to $\mathbb{Q}^{2}(x_i)$ (see Exercise 2.9 of \cite{DajToj}, or Theorem 1 of \cite{Yau}). In other words, $\Sigma$ meets an umbilical $\mathbb{Q}^2(b_i)$ orthogonally along the boundary points of $\partial_i \Sigma$ in this region, as we wanted to prove. 
	\end{proof}
	
\subsection{The free boundary correspondence} \label{subsectionfreeboundarycorrespondence} Theorem A implies a version of Lawson's correspondence for CMC surfaces in three-dimensional space forms meeting umbilical surfaces orthogonally.

	\begin{thmE}
		Let $(c,H,\{b_i\})$, $(\tilde{c},\tilde{H},\{\tilde{b}_i\})$ be ordered real numbers such that $c+H^2=\tilde{c}+\tilde{H}^2$ and $b_i-c=\tilde{b}_i-\tilde{c}\geq 0$ for all $i$. If a simply connected surface with boundary $(\Sigma^2,g)$ admits a constant mean curvature $H$ isometric immersion into $\mathbb{Q}^3(c)$ meeting a collection of umbilical surface $\mathbb{Q}^2(b_i)$ orthogonally along a boundary component $\partial_i \Sigma$, then it also admits a constant mean curvature $\tilde{H}$ isometric immersion into $\mathbb{Q}^3(\tilde{c})$ meeting a collection of umbilical $\mathbb{Q}^2(\tilde{b}_i)$ orthogonally along the boundary component $\partial_i \Sigma$.
	\end{thmE}
	
\begin{exam}
	Free boundary minimal surfaces of revolution with free boundary in the Euclidean sphere of radius one centred at the origin have also been classified: up to rotations, there is only one piece of catenoid with this property. This piece is called the \textit{critical catenoid}. According to the above correspondence, the critical catenoid is also isometrically immersed into the hyperbolic space as a piece of a catenoid cousin that meets (possibly different) \textit{horospheres} orthogonally along each boundary component. 
\end{exam}
		
		\subsection{Higher dimensions} Let $(\Sigma^n,g)$ denote a connected, oriented Riemannian surface with non-empty boundary $\partial \Sigma$, not necessarily complete. The higher dimensional generalisation of Theorem B reads as follows: 

\begin{thmD}
		Assume $(\Sigma^n,g)$ is isometrically immersed into $\mathbb{Q}^{n+1}(c)$ as a minimal hypersurface that intersects a collection of umbilical $\mathbb{Q}^{n}(b_i)$ orthogonally along the boundary components $\partial_i \Sigma$, and denote by $A$ its second fundamental form. Then, each boundary component $\partial_i \Sigma$ is umbilical in $(\Sigma^n,g)$ with second fundamental form $B_i=\pm\sqrt{b_i-c} \,g$ with respect to the outward pointing unit conormal $\nu$. Moreover, $A(X,\nu)=0$ for all vector fields $X$ tangent to $\partial \Sigma$. Finally, the symmetric two-tensor $\overline{g} = c(n-1)g - Ric_g$ is non-negative and, wherever it is positive definite, it is a Riemannian metric such that
		\begin{itemize}
			\item[$i)$] $\overline{g}(\overline{\nabla}_X Y,Z) = g(\nabla_X (\hat{A}(Y)),\hat{A}(Z))$ for every $X$, $Y$, $Z\in \mathcal{X}(\Sigma)$,
			\item[$ii)$] $Rm_{\overline{g}} = \frac{1}{2}\overline{g}\odot \overline{g} + \frac{c}{2} A\odot A$.
		\end{itemize}
		
		Conversely, assume that $(\Sigma^n,g)$ has boundary components $\partial_i \Sigma$ that are umbilical in $(\Sigma^n,g)$ with $B_i = \pm\sqrt{b_i -c} g$ for some constants $b_i\geq c$, that the symmetric tensor $\overline{g}=c(n-1)g-Ric_g$ is a Riemannian metric on $\Sigma^n$, and that there exists a smooth symmetric two-tensor $A$ on $\Sigma^n$, with $A(X,\nu)=0$ for every vector field $X$ tangent to $\partial \Sigma$, such that $\overline{g}=A\circ A$ satisfies conditions $i)$ and $ii)$.	Then $(\Sigma^n,g)$ admits, locally, minimal immersions into $\mathbb{Q}^{n+1}(c)$ that intersect a collection of umbilical $\mathbb{Q}^{n}(b_i)$ in right angles along boundary points in $\partial_i \Sigma$.
	\end{thmD}	

\begin{proof}
	$(\Rightarrow)$ The necessity of all these properties is clear from the proof of the boundary-less case (Theorem B) and Lemma \ref{lemma}. 
	
	$(\Leftarrow)$ Given the hypotheses, the proof of Theorem B guaranties that the pair $(g,A)$ satisfies the Gauss and Codazzi equations for a minimal isometric immersion into $\mathbb{Q}^{n+1}(c)$. Thus, every point of $\Sigma^n$, including boundary points, belongs to a simply connected neighbourhood that admits a minimal isometric immersion into $\mathbb{Q}^{n}(c)$. It remains only to understand how the boundary components $\partial_i \Sigma$ are mapped inside $\mathbb{Q}^{n+1}(c)$ by these immersions.
	
	Denote by $N$ the unit normal so that $A=\langle D_{-}N,{-}\rangle$. The unit conormal $\nu$ of $\partial_i \Sigma$ in $\Sigma^n$ and the unit normal $N$ of $\Sigma^n$ along $\partial_i \Sigma$ form an orthonormal frame of the normal bundle of $\partial_i \Sigma$ in $\mathbb{Q}^{n+1}(c)$. The section $\nu$ satisfies two properties:
	
	$a)$ $\nu$ is parallel with respect to the normal connection. In fact, by assumption,
	\begin{equation*}
		\langle D_X \nu,N\rangle = -\langle D_X N,\nu\rangle = -A(X,\nu) = 0
	\end{equation*}	
	for all vector fields $X$ tangent to $\partial_i \Sigma$. Also, $\langle \nabla_X \nu,\nu\rangle =X|\nu|^2/2 = 0$, so that there is no component of $\nabla_X \nu$ in the direction of $\nu$ either. Hence, the normal component of $\nabla_X \nu$ is zero for all such vector fields $X$.
	
	$b)$ By the assumption that $\partial_i \Sigma$ is an umbilical hypersurface of $(\Sigma^n,g)$, for all vector fields $X$ and $Y$ tangent to $\partial_i \Sigma$,
	\begin{equation*}
		\langle D_X \nu ,Y \rangle = g(\nabla_X^{\Sigma} \nu,Y) = \pm \sqrt{b_i -c}g(X,Y) = \pm \sqrt{b_i -c}\langle X,Y\rangle.
	\end{equation*}
	
	Standard results now imply that $\partial_i \Sigma$ lies inside an umbilical $\mathbb{Q}^{n}(x_i)$ in $\mathbb{Q}^{n+1}(c)$, whose sectional curvature is $x_i = c + (\pm\sqrt{b_i-c})^2 = b_i$, and moreover $\nu$ is normal to $\mathbb{Q}^{n}(x_i)$ (see again Exercise 2.9 of \cite{DajToj} and Theorem 1 of \cite{Yau}). In other words, under this local isometric immersion, $\Sigma^n$ meets an umbilical $\mathbb{Q}^n(b_i)$ orthogonally along the points corresponding to the boundary component $\partial_i \Sigma$.
\end{proof}

\section{Other intrinsic properties}
	
	To conclude this article, we would like to mention briefly two notions related to intrinsic properties of minimal surfaces in $\mathbb{R}^3$ meeting spheres of radius one orthogonally along their boundary components, from their inside.	
	
\subsection{Ricci surfaces with boundary}\label{subsecRicci}

	Let $(\Sigma^2,g)$ be a surface with negative Gaussian curvature $K$. It is straighforward to check that the equation that says $\sqrt{-K}g$ is a flat metric, namely
	\begin{equation*}
		-\Delta \log(-K) + 4K = 0,
	\end{equation*}
	is equivalent to the equation
	\begin{equation} \label{eqMoroianu}
		-K\Delta K + |\nabla K|^2 + 4K^3 = 0 .
	\end{equation}
	 But there is a difference between the two equations: the latter makes sense even where $K=0$. In \cite{MorMor}, Andrei Moroianu and Sergiu Moroianu noticed that, somehow, equation \eqref{eqMoroianu} contained more information than the original Ricci condition, in the sense that it forces a specific behaviour of its solutions near a zero. (Beware their different convention for the Laplacian). As a consequence, they were able to prove a very satisfactory equivalence:

	\begin{thm}[\textit{cf}. Theorem 2 in \cite{MorMor}]\label{thmMorMor}
		The following assertions about a Riemannian surface $(\Sigma^2,g)$ (without boundary) are equivalent.
		\begin{itemize}
			\item[$i)$] $(\Sigma^2,g)$ admits local minimal isometric immersions into $\mathbb{R}^3$.	
			\item[$ii)$] $(\Sigma^2,g)$ has non-positive Gaussian curvature $K$, which is a solution to equation \eqref{eqMoroianu}.
		\end{itemize}
	\end{thm}	

	While $i)\Rightarrow ii)$ follows directly from the proof of Theorem A using Gauss equation and equation \eqref{eqSimons}, the converse, as alluded before, requires a careful study of the zero set of solutions to equation \eqref{eqMoroianu}, which ruled out behaviours like that of Lawson's example (\textit{Cf}. Remark \ref{examLawson}).

	Because of this and other results, Moroianu and Moroianu proposed to call \textit{Ricci surfaces} the Riemannian surfaces whose Gaussian curvature satisfies equation \eqref{eqMoroianu}, and initiated the systematic investigation of these objects.

	The works of Yiming Zang \cite{Zan} and Benoît Daniel and Yiming Zang \cite{DanZan} developed the concept further, the first by imposing conditions on the ends of complete non-compact Ricci surfaces that match the intrinsic properties of the ends of finite total curvature minimal surfaces, the second by investigating a generalised notion of Ricci surface which captures, among other things, the key properties of the intrinsic geometry of other types of immersed CMC surfaces in space forms. \\

	In the case of surfaces with boundary, one would like to combine equation \eqref{eqMoroianu} with a geometrically meaningful boundary condition. The ideas we have been exploring in this article suggest clearly that the following boundary condition fulfils such requirement.
	
	\begin{defi}\label{defMoroianuboundary}
		Let $(\Sigma^2,g)$ be a Riemannian surface with boundary. We say that $(\Sigma^2,g)$ is a \textit{Ricci surface with boundary} when its boundary $\partial \Sigma$ has constant geodesic curvature one and its Gaussian curvature $K$ satisfies:
		\begin{equation*}
			\begin{cases}
				-K\Delta K + |\nabla K|^2 + 4K^3 = 0 \quad \text{in} \quad \Sigma, \\
			\frac{\partial K}{\partial \nu} = -4K \quad \text{in} \quad \partial \Sigma.
			\end{cases}
		\end{equation*}
	\end{defi}		
	For instance, consider a Riemannian surface $(\Sigma^2,g)$ with boundary that is isometrically immersed as a minimal surface in $\mathbb{R}^3$ in such way that it meets a collection of unit spheres orthogonally along their boundary components, from their inside. According to Lemma \ref{lemma} and \ref{lemmaBen} (recall in particular equation \eqref{eq|A|bdry} for $|A|^2=-2K$), $(\Sigma^2,g)$ is a \textit{Ricci surface with boundary} in the sense of the Definition above.
	
	We believe that a statement similar to Theorem \ref{thmMorMor} should also be true. But to establish it would lead us much further than we would like in this article, and so we leave it here.
	
	In the absence of an Enneper-Weierstrass representation for free boundary minimal surfaces inside the unit ball, we think that the study of compact Ricci surfaces with boundary in the sense of Definition \ref{defMoroianuboundary} could be fruitful. (For instance, Jaehoon Lee and Eungbeom Yeon adopted a similar, but distinct perspective in \cite{LeeYeo}).
	\begin{rmk}
		In \cite{DomRonVit}, the authors classified surfaces of revolution in $\mathbb{R}^3$ that are Ricci surfaces in the sense of \cite{MorMor}. They also studied the intersection of these surfaces with the Euclidean sphere of radius one centred at the origin, and found a 1-parameter family of such surfaces interpolating between the critical catenoid and its axis of rotation. None of them, except the critical catenoid, are Ricci surfaces with boundary in the sense of Definition \ref{defMoroianuboundary}.
	\end{rmk}

\subsection{Steklov eigenvalues} 

	If $\Sigma^2$ is a compact minimal surface with boundary in $\mathbb{R}^3$ meeting the unit sphere centred at the origin orthogonally, then its coordinate functions $x^1,x^2,x^3$ satisfy:
	\begin{equation*}
		\begin{cases}
			\Delta x^i = 0 \quad \text{on} \quad \Sigma, \\
			\frac{\partial x^i}{\partial \nu} = x^i \quad \text{in} \quad \partial \Sigma. 
		\end{cases}
	\end{equation*} 
	This means that the three coordinate functions are \textit{Steklov eigenfunctions} of $(\Sigma^2,g)$, associated to the eigenvalue $1$. This observation was a starting point for the work of Ailana Fraser and Rick Schoen \cite{FraSch}, who sparked a renovated interest in the theory of compact free boundary minimal surfaces, by connecting it to  a spectral maximisation problem for the Steklov (or Dirichlet-to-Neumann) operator. 
	
	The relation between the spectral theory of the Steklov operator and compact minimal surfaces with free boundary is not restricted to the Euclidean case. Vanderson Lima and Ana Menezes \cite{LimMen} identified the right spectral problem that describes minimal surfaces with free boundary in geodesic spheres of the sphere $\mathbb{S}^3$. An interesting new feature is that this problem involves a combination of the first two Steklov eigenvalues. Later, Vladimir Medvedev, among other things, extended the work done in \cite{LimMen} to the case of minimal surfaces with free boundary in geodesic balls of the hyperbolic space $\mathbb{H}^3$. (See also the more general setup proposed Romain Petrides \cite{Pet}).
	
	The free boundary correspondence discussed in Section \ref{subsectionfreeboundarycorrespondence} suggests that such spectral relationships - which are, after all, part of the intrinsic geometry of these immersed surfaces with boundary - are even more widespread than earlier thought. They could be useful, for instance, in the study of compact constant mean curvature $H= 1$ surfaces with boundary in $\mathbb{H}^3$ that meet horospheres orthogonally. \\
		
\noindent \textbf{Acknowledgements:} I thank Ruy Tojeiro, who had the excellent idea of commemorating Professor Marcos Dajczer seventy-fifth birthday with an Festchrift volume, for the kind invitation to contribute to it with an article. I would also like to express my appreciation of my collaborators Alessandro Carlotto and Ben Sharp, with whom I enjoyed, over already eight years of scientific collaboration and friendship, so many interesting conversations about free boundary minimal surfaces. Several ideas that I exposed in this article were certainly born out of these conversations.

\section*{Appendix}
	Consider the following subsets of the complex plane $\mathbb{C}$,
	\begin{equation*}
		\mathbb{D} = \{ z=x+iy\in \mathbb{C}\,|\,|z|<1\} \quad \text{and} \quad \mathbb{D}_+ = \{z=x+iy \in \mathbb{D}\,|\, y> 0\},
	\end{equation*}		
	and let $\Gamma$ denote the horizontal boundary of $\mathbb{D}_+$,
	\begin{equation*}
		\Gamma = \{z=x+iy\in \overline{\mathbb{D}}_+\,|\, y=0\}.
	\end{equation*}
	
	The following elementary lemma was an important ingredient of the proof of Theorem A and Theorem C. (\textit{Cf}. Lemma 4.2 in \cite{MorMor}).	
	
	\begin{lemm}\label{lemlogharmonic}\hfill
		\begin{itemize}
			\item[$i)$] Let $F : \mathbb{D}\rightarrow \mathbb{R}$ be a positive smooth function such that $\log(F)$ is harmonic on $\mathbb{D}$. Then there exists a holomorphic function $h:D \rightarrow \mathbb{C}$ such that $F = |h|^2$.
			\item[$ii)$] Let $F : \mathbb{D}_+\cup \Gamma\rightarrow \mathbb{R}$ be a positive smooth function such that $\log(F)$ is harmonic on $\mathbb{D}_+$ and satisfies the Neumann boundary condition on $\Gamma$. Then there exists a holomorphic function $h:\mathbb{D} \rightarrow \mathbb{C}$ that has zero imaginary part on $\Gamma$ and is such that $F=|h|^2$ on $\mathbb{D}_+\cup \Gamma$.
		\end{itemize}
	\end{lemm}
	
	\begin{proof}
		$i)$ Recall that $\Delta = 4\partial_{\overline{z}}\partial_z$. Thus, the assumption is that $\partial_z \log(F)$ is a holomorphic function on $\mathbb{D}$.
		
		Integrating over paths based at $0$, we obtain a holomorphic function $g_1 :\mathbb{D}\rightarrow \mathbb{C}$ such that $\partial_z g_1 = \partial_z \log(F)$. Then, the function $g_2=\overline{\log(F)-g_1}$ is holomorphic, because
	\begin{equation*}
		\partial_{\overline{z}}g_2 = \overline{\partial_z(\log(F)-g_1)}=0.
	\end{equation*}	
	Hence $\log(F)=g_1+\overline{g}_2$. The imaginary part of $g_1$ must be equal to the imaginary part of $g_2$, because $\log(F)$ is real. Furthermore, $\log(F)$ equals the real part of the sum $g_1+g_2$ as well. Therefore $h=\exp((g_1+g_2)/2)$ is a holomorphic function such that $|h|^2=F$.
			
		$ii)$ By virtue of the boundary condition, it is possible to extend $F$ to $\mathbb{D}$ by reflection across $\Gamma$, so that we may assume that $\log(F)$ is a harmonic function on $\mathbb{D}$ invariant by the reflection across $\Gamma$. Applying the same reasoning as in case $i)$, we obtain holomorphic functions $g_1$ and $g_2$ with the same imaginary part in such way that $\partial_z g_1 = \partial_z \log(F)$ and $\log(F)=g_1+\overline{g}_2$. 
				
		By assumption, $\partial_y \log(F) = 0$ on $\Gamma$. Hence,
		\begin{equation*}
			2\partial_z \log(F) = \partial_x \log(F) -i\partial_y \log(F) = \partial_x \log(F)
		\end{equation*}
is real on $\Gamma$. If we write $g_1=a+ib$ according to its real and imaginary part,
		\begin{equation*}
			2\partial_z g_1 = (\partial_x a + \partial_y b) + i(\partial_x b-\partial_y a).
		\end{equation*}
		Thus, the vanishing of the imaginary part of $\partial_z g_1 = \partial_z \log(F)$ on $\Gamma$ together with the Cauchy-Riemann equations satisfied by $g_1$ implies, in particular, that
		\begin{equation*}
			\partial_x b = \partial_y a = -\partial_x b \quad \text{on} \quad \Gamma \quad \Rightarrow \quad \partial_x b = 0 \quad \text{on} \quad \Gamma.
		\end{equation*}
		Hence $g_1$ has constant imaginary part on $\Gamma$. But then the imaginary part of $g_2$ is constant as well.
		
		In conclusion, the holomorphic function $g_1+g_2$ has constant imaginary part on $\Gamma$, say $k\in \mathbb{R}$. The function $h=\exp((g_1+g_2)/2)\exp(-ki/2)$ is therefore a holomorphic function, with squared norm equal to $F$, which moreover has zero imaginary part on $\Gamma$. And this is all we wanted to show.
	\end{proof}

\end{document}